\documentclass[draft,12pt,english]{article}

\usepackage{amsfonts,amsmath,amsthm,amscd,amssymb,latexsym,cite,verbatim,texdraw,floatflt,
	caption2,pb-diagram}
\usepackage[mathscr]{eucal}

\theoremstyle{plain}
\newtheorem{theorem}{Theorem}

\theoremstyle{definition}

\newcommand{\keywords}{\textbf{Key words. }\medskip}
\newcommand{\subjclass}{\textbf{MSC 2010. }\medskip}
\renewcommand{\abstract}{\textbf{Abstract. }\medskip}

\begin{document}

\sloppy

\title{\bf Estimates of deviations of Fourier sums\\ on Weyl--Nagy classes \boldmath{$W^r_{\beta,1}$}}
\author{\bf A.S.~Serdyuk, I.V.~Sokolenko}
\date{\ }

\maketitle

\begin{abstract}
We establish estimates for exact upper bounds of deviations of partial Fourier sums $S_{n-1}(f)$ on classes $W^r_{\beta,1}, r>2, \beta\in\mathbb{R},$ of $2\pi$-periodic functions whose $(r,\beta)$-derivatives in the Weyl--Nagy sense belong to the unit ball of the space $L_1$. The specified estimates allow us to write asymptotic equalities for the quantities \mbox{$\sup\limits_{f\in W^r_{\beta,1}}|f(x)-S_{n-1}(f;x)|$} as \mbox{$n\rightarrow\infty$,} $r\rightarrow\infty$ for arbitrary relations between the parameters $r$ and $n$.
\end{abstract}



\bigskip
\begin{flushright}
	\textit{Dedicated to the bright memory of \\ Vitaly Pavlovich Motorny}
\end{flushright}

\section{Introduction}
Let   $L_{p}$,
$1\le  p<\infty$, be a space of $2\pi$--periodic functions $\varphi$ sum\-mable  with $p$th power  on the interval $[-\pi,\pi)$ with the norm
\begin{equation}
	\|\varphi\|_{p}=\bigg(\int\limits_{-\pi}^{\pi}|\varphi(t)|^pdt\bigg)^{1/p},
\end{equation}
let $L_{\infty}$ be a space of measurable and essentially bounded   $2\pi$--periodic functions  $\varphi$ with the norm given by the equality
\begin{equation}
	\|\varphi\|_{L_\infty}=\mathop {\rm esssup}\limits_{t} |f(t)|,
\end{equation}
and let $C$ be a space of continuous $2\pi$--periodic functions  $\varphi$ with the norm
\begin{equation}
	\|\varphi\|_{C}=\max\limits_{t}|\varphi(t)|.
\end{equation}

Further, let $W^r_{\beta,p},$ $ r>0,\ \beta\in\mathbb{R},\ 1\le p\le \infty,$   be classes of  $2\pi$-periodic functions $f$ that can be represented in the form of convolution
\begin{equation}\label{1}
	f(x)=
	\frac{a_0}{2}+\frac{1}{\pi}\int\limits_{-\pi}^{\pi}
	\varphi(x-t) B_{r,\beta}(t)dt, \ \ \ a_0\in\mathbb R, \quad  \int\limits_{-\pi}^{\pi}\varphi(t) dt=0,
\end{equation}
with Weyl-Nagy kernels $B_{r,\beta}$  of the form
\begin{equation}\label{2}
	B_{r,\beta}(t)=\sum\limits_{k=1}^\infty k^{-r}\cos\left(kt-\frac{\beta\pi}2\right),\quad r>0,\quad \beta\in\mathbb{R},
\end{equation}
of function $\varphi$ satisfying the condition $\varphi\in B_p$, where
\begin{equation}\label{3}
	B_p:=\left\{\varphi\in L_p: \|\varphi\|_p\le1\right\},\quad1\le p \le\infty.
\end{equation}

The classes $W^r_{\beta,p}$ are called Weyl-Nagy classes (see, for example, \cite{Stepanets1995, Stepanets2005, Stechkin1956}), and the function $\varphi$ in the representation \eqref{1} is called the $(r,\beta)$-derivative of the function $f$ in the Weyl-Nagy sense  and denoted by $f^r_\beta$.

If $r\in\mathbb N$ and $ \beta=r,\ $ then functions of the form (\ref{2}) are well-known Bernoulli kernels and the corresponding classes $W^r_{\beta,p}$ coincide with the known classes $W^r_{p}$ of $2\pi$-periodic functions $f$ which have absolutely continuous derivatives up to $(r-1)$-th order, inclusivly, such that $\|f^{(r)}\|_p\le1$. Moreover, the equality $f^{(r)}(x)=f^r_r(x), \ r\in\mathbb{N}$ is true almost everywhere. If $r>0$ and $\beta=r$, then the classes $W^r_{\beta,p}$ are known Weyl classes $W^r_p$.

For all $1<p,s<\infty$, $r>(\frac1p-\frac1s)_+=\left\{\begin{array}{ll}
	\frac1p-\frac1s,& p<s,\\
	0,& p\ge s,
\end{array}\right.$ and $\beta\in\mathbb{R}$ the embedding $W^r_{\beta,p}\subset L_s$ is true. In addition, for any $1\le p\le\infty$, $r>\frac1p,$ $\beta\in\mathbb{R}$ we have the embeddings $W^r_{\beta,p}\subset C$ (see, for example, \cite[\S5.4]{Stepanets1995}, \cite[\S6.6]{Stepanets2005}). In this paper we consider the case $r>2$ and, therefore, the Weyl-Nagy classes in this situation consist only of continuous functions.

Let $f\in L$. By $S_{n-1}(f)=S_{n-1}(f;x)$ we denote the partial Fourier sum of order $n-1$ of the function $f$:
\begin{equation}\label{3a}
	S_{n-1}(f;x)=\frac{a_0}2+ \sum\limits_{k=1}^{n-1}(a_k\cos kx+b_k\sin kx),
\end{equation}
where $a_k$ and $b_k$ are the Fourier coefficients of the function $f$:
$$
a_k=\frac1\pi\int\limits_{-\pi}^{\pi}f(t)\cos ktdt,\quad b_k=\frac1\pi\int\limits_{-\pi}^{\pi}f(t)\sin ktdt.
$$
Let us denote by $\rho_n(f;x) $ the deviation from the function $f(x)$ of its partial Fourier sum $S_{n-1}(f;x)$, i.e.
$$
\rho_n(f;x) :=f(x)-S_{n-1}(f;x).
$$

Let $\mathfrak N$ be some functional class from the space $C$.
The quantity
\begin{equation}\label{4}
	{\cal E}_n(\mathfrak N):=\sup\limits_{f\in \mathfrak N}|\rho_n(f;x)|=\sup\limits_{f\in \mathfrak N}|f(x)-{S}_{n-1}(f;x)|
\end{equation}
is considered and the problem of finding asymptotic equality for the quantity $ {\cal E}_n(\mathfrak N)$  is investigated as $n\rightarrow\infty$. Note that if the class $\mathfrak{N}$ is invariant with respect to the shift of the argument (i.e., if $f(x)$ belongs to the class $\mathfrak{N}$, then for any $h\in\mathbb{R}$ the function $f_h(x)=f(x+h)$ also belongs to the class $\mathfrak{N}$), then the quantity $ {\cal E}_n(\mathfrak N)$ of the form \eqref{4} does not depend on the value of $x$.

In this paper we investigate the asymptotic behavior of quantities $ {\cal E}_n(\mathfrak N)$ of the form \eqref{4} in the case when the Weyl-Nagy classes $W^r_{\beta,1}, r>2, \beta\in\mathbb{R}$ play the role of $\mathfrak{N}$. The central result is theorem \ref{3t} (see section 3), which allows us to obtain uniform asymptotic equalities for the quantities ${\cal E}_{n}(W^r_{\beta,1})$ as $n\rightarrow\infty$ and $r\rightarrow\infty$ depending on the relations between the parameters $r$ and $n$.

\section{Historical review}

The problem of finding the order of decrease of quantities of the form \eqref{4} on classes of differentiable functions was actually investigated by H.~Lebesgue \cite{1910-Lebesgue}. Asymptotic equalities as $n\rightarrow\infty$ for the quantities ${\cal E}_{n}(W^r_{\beta,\infty}), r\in\mathbb{N},$ were first established by A.M. Kolmogorov \cite{Kolmogorov1935}. Kolmogorov's research was continued by numerous researchers, including V.T.~Pinkevich \cite{Pinkevich1940}, S.M.~Nikol'skii \cite{Nikol'skii1941,Nikol'skii1946}, A.V.~Efimov \cite{Efimov1960}, S.B.~Stechkin \cite{Stechkin1956,Stechkin1980}, S.O.~Telyakovskii \cite{Telyakovskii1968}, O.I.~Stepanets (see monographs \cite{Stepanets1995,Stepanets2005} and the bibliography given in them), V.P.~Motorny \cite{Motornyi1974} and others.

For the Weyl-Nagy classes $W^r_{\beta,\infty}$ for all fixed $r>0$ and $\beta\in\mathbb{R}$ the asymptotic equality
\begin{equation}\label{4'}
	{\cal E}_{n}(W^r_{\beta,\infty})=\frac4{\pi^2}\frac{\ln n}{n^r}+O\left(\frac1{n^r}\right)
\end{equation}
holds as $n\rightarrow\infty$ (see, for example, \cite[\S3.7]{Stepanets1995}, \cite[\S5.3]{Stepanets2005}, \cite{Stechkin1980}). In the formula \eqref{4'} the quantity $O$ is uniformly bounded in the parameter $n$ but may depend on $r$ and $\beta$.

The rate of decrease of remainder term in the esyimate \eqref{4'} depending on the parameters $r$ and $\beta$ was studied in the works
\cite{Natanson1961,Selivanova1955,Sokolov1955,Stechkin1980,Telyakovskii1968, Telyakovskii1989, Telyakovskii1994} and others.

In particular, in the work of S.A.~Telyakovskii \cite{Telyakovskii1968} a uniform in  $n\in\mathbb{N}$, $r>0$ and $\beta\in\mathbb{R}$  estimate
\begin{equation}\label{4d}
	{\cal E}_{n}(W^r_{\beta,\infty})=\frac4{\pi^2}\dfrac{1}{n^r}\ln\frac{n}{\min\{n,r+1\}}+\dfrac{2}{\pi r}\left|\sin\frac{\beta\pi}{2}\right|\dfrac{1}{n^r}+\mathcal{O}\left(\frac1{n^r}\right)
\end{equation}
was established. This estimate is the asymptotic equality under the condition $r=o(n)$ as $n\rightarrow\infty$. If this condition is not met, then the formula \eqref{4d} provides only an upper order estimate.

Asymptotic equalities for the quantities ${\cal E}_{n}(W^r_{\beta,\infty})$ as $n\rightarrow\infty$ and $r\rightarrow\infty$ and for all possible relations between the parameters $r$ and $n$ were established by S.B.~Stechkin \cite{Stechkin1980}. Namely, for $r\ge1$ and $\beta\in\mathbb{R}$ he proved that
\begin{equation}\label{6}
	{\cal E}_{n}(W^r_{\beta,\infty})=\frac8{\pi^2}\frac{\mathbf{K}(e^{-r/n})}{n^r}+\dfrac{\mathcal{O}(1)}{rn^r},
\end{equation}
where
\begin{equation}\label{7}
	\mathbf{K}(q)=\int\limits_{0}^{\pi/2}\frac{dt}{\sqrt{1-q^2\sin^2t}}
\end{equation}
is the complete elliptic integral of the first kind and $\mathcal{O}(1)$ is a quantity uniformly bounded in all considered parameters.

The investigations of the exact or asymptotically exact behavior of the quantities ${\cal E}_{n}(W^r_{\beta,p})$ for $p\in[1,\infty)$ were carried out in the works \cite{Babenko-Pichugov1980,Serdyuk_Sokolenko2011,Serdyuk_Sokolenko2013,Serdyuk_Sokolenko2019MFAT,Serdyuk_Sokolenko2020UMB,Serdyuk_Sokolenko2022UMJ} and others.

In this work we consider the case $p=1$ and continue  our research \cite{Serdyuk_Sokolenko2019MFAT} and \cite{Serdyuk_Sokolenko2022UMJ}, with which it is closely related. Thus, in the work \cite{Serdyuk_Sokolenko2019MFAT}  it was shown, in particular, that for $r\ge n+1$  a uniform on all   parameters estimate
\begin{equation}\label{4g}
	{\cal E}_{n}(W^r_{\beta,1})=
	\frac1{n^r}\left(\frac1{\pi} +\mathcal{O}(1)\left(1+\frac1n\right)^{-r}\right)
\end{equation}
holds (see Theorem 1 for $p=1$ from \cite{Serdyuk_Sokolenko2019MFAT}).

Later in the work \cite{Serdyuk_Sokolenko2022UMJ} it was proved that under the condition \begin{equation}\label{4z}
	\sqrt{n}+1\le r\le n+1
\end{equation}
the estimate
\begin{equation}\label{4k}
	{\cal E}_{n}(W^r_{\beta,1})=
	\frac1{n^r}\left( \frac1{\pi(1-e^{-r/n})}
	+\mathcal{O}(1)\frac n{r^{2}}\right),
\end{equation}
holds (see Theorem 1 for $p=1$ from \cite{Serdyuk_Sokolenko2022UMJ})
and under the condition
\begin{equation}\label{4l}
	n+1\le r\le n^2
\end{equation}
the estimate
\begin{equation}\label{4m}
	{\cal E}_{n}(W^r_{\beta,1})=
	\frac1{n^r}\left( \frac1{\pi(1-e^{-r/n})}
	+\mathcal{O}(1)\dfrac r{n^2}e^{-r/n}\right).
\end{equation}
holds  (see Theorem 2 for $p=1$ from \cite{Serdyuk_Sokolenko2022UMJ}).
In formulas \eqref{4k} and \eqref{4m} the quantities $\mathcal{O}(1)$ are uniformly bounded in all considered parameters.

The case
\begin{equation}\label{4n}
	r\ge n^2
\end{equation}
is covered by the formula \eqref{4g}. However, with the indicated relations between $r$ and $n$, it can be written in a form similar to \eqref{4k} and \eqref{4m}. Indeed, under the condition \eqref{4n}, the following formulas
\begin{equation}\label{4o'}
	\dfrac1{1-e^{-r/n}}=1+\mathcal{O}(1)e^{-r/n}
\end{equation}
and
\begin{equation}\label{4o}
	e^{-r/n}\le \left(1+\frac1n\right)^{-r}\le e^{-r/(n+1)}
\end{equation}
hold.

Then, by virtue of \eqref{4g}, \eqref{4o'} and \eqref{4o}, we have that, under the condition \eqref{4n}, the estimate
\begin{equation}\label{4r}
	{\cal E}_{n}(W^r_{\beta,1})=
	\frac1{n^r}\left( \frac1{\pi(1-e^{-r/n})}
	+\mathcal{O}(1)\left(1+\frac1n\right)^{-r}\right),
\end{equation}
is true. Here, the quantity $\mathcal{O}(1)$ is uniformly bounded in all considered parameters.

The main approach, which allowed us to obtain the estimates \eqref{6}, \eqref{4k}, \eqref{4m} and \eqref{4r}, consists in approximating in $L_{p'}$-metrics the remainder $\rho_n(B_{r,\beta};\cdot)$ of Weyl-Nagy kernels of the form \eqref{2} by the corresponding remainder $\rho_n(P_{q,\beta};\cdot)$ of Poisson kernels
\begin{equation}\label{!}
	P_{q,\beta}(t)=\sum\limits_{k=1}^\infty q^{-k}\cos\left(kt-\frac{\beta\pi}2\right),\quad q\in(0,1),\quad \beta\in\mathbb{R}.
\end{equation}
As a result of using this approach, which was first proposed by S.B.~Stechkin in the work \cite{Stechkin1980} for $p=\infty$, the problem of strong asymptotics of the quantities ${\cal E}_{n}(W^r_{\beta,p})$ is reduced to the corresponding problem for the quantities ${\cal E}_{n}(C^q_{\beta,p})$, where $C^q_{\beta,p},\ q\in(0,1),\ \beta\in\mathbb{R},\ 1\le p\le\infty,$ are classes of Poisson integrals of the form
\begin{equation}\label{!!}
	C^q_{\beta,p}=\left\{ f(x)= \frac{a_0}{2}+\frac{1}{\pi}\int\limits_{-\pi}^{\pi}
	\varphi(x-t) P_{q,\beta}(t)dt, \ a_0\in\mathbb R, \int\limits_{-\pi}^{\pi}\varphi(t) dt=0, \varphi\in B_p\right\},
\end{equation}
where the kernels $P_{q,\beta}(\cdot)$ are denoted by the equality \eqref{!}, and the sets $B_p$ are denoted by the equality \eqref{3}.

As for the problem of finding the asymptotics of quantities of the form \eqref{4} for classes of Poisson integrals $C^q_{\beta,p}$ and their generalizations, it has been completely solved thanks to the works of \cite{Nikol'skii1946,Serdyuk_2005_8,Serdyuk_2005_10,Serdyuk_Stepanyuk2019Anal,Serdyuk_Stepanyuk2023UMJ,Stechkin1980}.

At the same time, the case
\begin{equation}\label{4s}
	2<r\le \sqrt n+1
\end{equation}
has remained unexplored until now.

\section{Main Results}

The following statement allows us to write the principal term of the asymptotic expansion of the quantity ${\cal E}_{n}(W^r_{\beta,1})$ as $r\rightarrow\infty$ and $r/n\rightarrow0$.  This case also fully covers the relation \eqref{4s} as $r\rightarrow\infty$.

\begin{theorem}\label{1t}
	Let $r>2,$ $\beta\in\mathbb R$ and $n\in\mathbb{N}$. Then the formula
	\begin{equation}\label{1t1}
		{\cal E}_{n}(W^r_{\beta,1})=\frac1{n^r}\frac nr\left(
		\frac1{\pi} +\bar\theta\left(\frac{1}{r-2}+\frac rn\right)\right),
	\end{equation}
	is true. Here $\bar\theta=\bar{\theta}_{r,n,\beta}\in\left(-\frac2\pi,\frac2\pi\right)$.
\end{theorem}

\begin{proof}[Proof of Theorem~\ref{1t}]
	For $r>2$ from Theorem 2.2 of the work \cite{Serdyuk_Stepanyuk2023UMJ} it follows that
	\begin{equation}\label{1td1}
		{\cal E}_{n}(W^r_{\beta,1})=
		\frac{1}{\pi} \left(\sum\limits_{k=n}^{\infty}\frac1{k^{r}}+ \frac{\theta}{n} \sum\limits_{k=1}^{\infty}\frac k{(k+n)^r}\right),
	\end{equation}
	where $\theta\in[-1,0]$.
	
	Because
	\begin{equation}\label{1td2}
		\frac n{r-1}\frac1{n^r}= \int\limits_n^\infty\frac1{t^r}dt<\sum\limits_{k=n}^{\infty}\frac1{k^r}<\frac1{n^r}+\int\limits_n^\infty\frac1{t^r}dt=\left(1+\frac n{r-1}\right)\frac1{n^r},
	\end{equation}
	then
	\begin{equation}\label{1td4}
		\sum\limits_{k=n}^{\infty}\frac1{k^r}=\frac1{n^r}\left(\frac n{r-1}+\theta_1\right),\quad\theta_1\in[0,1].
	\end{equation}
	
	Taking into account the estimate \eqref{1td4} and the following equalities
	\begin{equation}\label{1td5}
		\frac1n\sum\limits_{k=1}^{\infty}\frac k{(k+n)^r}=\frac1n\sum\limits_{k=0}^{\infty}\frac {k+n}{(k+n)^r}-\sum\limits_{k=0}^{\infty}\frac 1{(k+n)^r}=
		\frac1n\sum\limits_{k=n}^{\infty}\frac 1{k^{r-1}}-\sum\limits_{k=n}^{\infty}\frac 1{k^r},
	\end{equation}
	we obtain estimates
	\begin{equation}
		\frac1n\sum\limits_{k=n}^{\infty}\frac 1{k^{r-1}}=\frac1n\frac1{n^{r-1}}\left(\frac n{r-2}+\theta_1\right)=\frac1{n^r}\left(\frac n{r-2}+\theta_1\right),\quad\theta_1\in[0,1],
	\end{equation}
	\begin{equation}\label{1td6}
		\frac1n\sum\limits_{k=n}^{\infty}\frac 1{k^{r-1}}-\sum\limits_{k=n}^{\infty}\frac 1{k^r}=\frac1{n^r}\left(\frac n{r-2}-\frac n{r-1}+\theta_2\right)
		=\frac1{n^r}\left(\frac n{(r-1)(r-2)}+\theta_2\right),
	\end{equation}
	$\theta_2\in[-1,1]$.
	
	Combining the formulas \eqref{1td4}--\eqref{1td6} allows you to write
	$$
	\sum\limits_{k=n}^{\infty}\frac 1{k^{r}}+\frac{\theta}n\sum\limits_{k=1}^{\infty}\frac k{(k+n)^r}
	=\frac1{n^r}\left(\frac n{r-1}+\theta_1\right)+\frac{\theta}{n^r}\left(\frac n{(r-1)(r-2)}+\theta_2\right)
	=
	$$
	\begin{equation}\label{1td7}
		=\frac1{n^r}\left(\frac n{r-1}+\frac {\theta n}{(r-1)(r-2)}+\theta_1+\theta\theta_2\right)
		=\frac1{n^r}\left(\frac n{r-1}\left(1+\frac \theta{r-2}\right)+\theta_3\right),
	\end{equation}
	where $\theta_3\in[-1,2]$.
	
	Because
	$$
	\frac n{r-1}=\frac nr-\frac n{r(r-1)},
	$$
	then
	\begin{equation}\label{12a}
		\frac n{r-1}\left(1+\frac \theta{r-2}\right)=\frac nr-\frac n{r(r-1)}+\frac {n\theta}{r(r-2)}-\frac{n\theta}{r(r-1)(r-2)}=\frac nr(1+A_{n,r}),
	\end{equation}
	where
	\begin{equation}\label{12b}
		A_{n,r}=A_{n,r}(\theta):=-\frac1{r-1}+\frac {\theta}{r-2}-\frac{\theta}{(r-1)(r-2)},\quad\theta\in[-1,0].
	\end{equation}
	
	Due to \eqref{12b}
	$$
	A_{n,r}\le-\frac1{r-1}+\frac1{(r-1)(r-2)}=-\frac1{r-1}+\frac1{r-2}-\frac1{r-1}\le\frac1{r-2},
	$$
	$$
	A_{n,r}\ge-\frac1{r-1}-\frac1{r-2}>-\frac2{r-2},
	$$
	and therefore
	\begin{equation}\label{12v}
		A_{n,r}=\frac{\theta_4}{r-2},\quad\theta_4\in(-2,1).
	\end{equation}
	
	From formulas \eqref{12a} and \eqref{12v} we obtain
	\begin{equation}\label{12g}
		\frac n{r-1}\left(1+\frac\theta{r-2}\right)=\frac nr\left(1+\frac{\theta_4}{r-2}\right),\quad\theta_4\in(-2,1).
	\end{equation}
	
	Taking into account \eqref{12g}, we have
	\begin{equation}\label{12d}
		\frac1{n^r}\left(\frac n{r-1}\left(1+\frac\theta{r-2}\right)+\theta_3\right)=\frac1{n^r}\left(\frac nr\left(1+\frac{\theta_4}{r-2}\right)+\theta_3\right),
	\end{equation}
	$\theta_3\in[-1,2],\ \theta_4\in(-2,1).$
	
	Therefore, from \eqref{1td1}, \eqref{1td7} and \eqref{12d}, it follows that
	\begin{equation}\label{12g'}
		{\cal E}_{n}(W^r_{\beta,1})=
		\frac 1{\pi n^r}\left(\frac{n}{r}\left(1+\frac{\theta_4}{r-2}\right)+\theta_3\right),\quad\theta_3\in[-1,2],\ \theta_4\in(-2,1),
	\end{equation}
	and, consequently, a somewhat rougher in comparison with \eqref{12g'} estimate  \eqref{1t1} follows.
	Theorem \ref{1t} is proved.
\end{proof}

The uniform in all parameters estimate \eqref{1t1} is an asymptotic equality as $r\rightarrow\infty$ and $r/n\rightarrow0$.

Since
\begin{equation}\label{1td10}
	\frac1{1-e^{-x}}= \frac1x+\mathcal{O}(1),\quad x>0,
\end{equation}
then
\begin{equation}\label{1td10}
	\frac1{1-e^{-r/n}}= \frac nr+\mathcal{O}(1).
\end{equation}

Furthermore, if the condition \eqref{4s} is satisfied, then
\begin{equation}\label{33a}
	\dfrac n{r(r-2)}\ge\dfrac{(r-1)^2}{r(r-2)}>1.
\end{equation}

Taking into account \eqref{1td10} and \eqref{33a}, from the formula \eqref{1t1} of Theorem \ref{1t} we obtain the following statement.

\begin{theorem}\label{2t}
	Let $r>2,$ $\beta\in\mathbb R$ and $n\in\mathbb{N}$. Then, under the condition $2<r\le \sqrt{n}+1$, the equality
	\begin{equation}\label{2t1}
		{\cal E}_{n}(W^r_{\beta,1})= \frac1{n^r}\left(\frac1{\pi(1-e^{-r/n})}
		+\mathcal{O}(1)\dfrac n{r(r-2)}\right),
	\end{equation}
	is true. Here, the quantity $\mathcal{O}(1)$ is uniformly bounded in all considered parameters.
\end{theorem}

Considering successively the cases \eqref{4s}, \eqref{4z}, \eqref{4l} and \eqref{4n} and associating with each of them the estimates \eqref{2t1}, \eqref{4k}, \eqref{4m} and \eqref{4r}, respectively, we otain the general statement.

\begin{theorem}[{\bf main theorem}]\label{3t}
	Let $r>2,$ $\beta\in\mathbb R$ and $n\in\mathbb{N}$. Then the formula
	\begin{equation}\label{3t1}
		{\cal E}_{n}(W^r_{\beta,1})= \frac1{n^r}\left(\frac1{\pi(1-e^{-r/n})}
		+\mathcal{O}(1)\delta_{r,n}\right),
	\end{equation}
	holds. Here,
	\begin{equation}\label{3t2}
		\delta_{r,1}:=e^{-r},\quad \delta_{r,n}:=
		\left\{
		\begin{array}{cl}
			\dfrac n{r(r-2)},& 2< r\le n+1,\\ \ \\
			\dfrac r{n^2}e^{-r/n},&n+1\le r\le n^2,\\ \ \\
			\left(1+\dfrac1n\right)^{-r},
			& \quad \ n^2\le r,\\
		\end{array}
		\right.
		n\ge 2.
	\end{equation}
	and $\mathcal{O}(1)$ is a quantity uniformly bounded in all considered parameters.
\end{theorem}

The formula \eqref{3t1} is in fact the corresponding for class $W^r_{\beta,1}$ analogue of Stechkin's estimate \eqref{6} obtained earlier for classes $W^r_{\beta,\infty}$.

Incidentally, we note that based on the estimate \eqref{1td1} used in the proof of the theorem \ref{1t}, we can write the estimate of the quantity ${\cal E}_{n}(W^r_{\beta,1})$ in the  integral form.

Let
\begin{equation}\label{6a}
	\zeta(s,q)=\sum\limits_{k=0}^\infty \frac1{(q+k)^{s}},\quad Re(s) > 1,\ Re(q) > 0,
\end{equation}
be the Hurwitz zeta function.

Given \eqref{6a}, we can write
\begin{equation}\label{6b}
	\sum\limits_{k=n}^\infty\frac{1}{k^r}= \zeta(r,n),
\end{equation}
\begin{equation}\label{6v}
	\frac1n\sum\limits_{k=1}^\infty\frac{k}{(k+n)^r}= \frac1n\zeta(r-1,n+1)-\zeta(r,n+1).
\end{equation}

Then the formula \eqref{1td1} will be written in the form
\begin{equation}\label{6g}
	{\cal E}_{n}(W^r_{\beta,1})=\frac1\pi\left(\zeta(r,n)+\theta R_{r,n}\right),\quad\theta\in[-1,0],
\end{equation}
where
\begin{equation}\label{38v}
	R_{r,n}=\frac1n\zeta(r-1,n+1)-\zeta(r,n+1).
\end{equation}

Since (see, for example, \cite[p.~1086, formula 9.511]{Gradshteyn_Ryzhik1963})
\begin{equation}\label{38zh}
	\zeta(s,q)=\frac1{\Gamma(s)}\int\limits_0^\infty\frac{t^{s-1}e^{-qt}}{1-e^{-t}}dt,\quad Re(s) > 1,\ Re(q) > 0,
\end{equation}
then we can write the estimate \eqref{6g} in integral form
\begin{equation}\label{38d}
	{\cal E}_{n}(W^r_{\beta,1})=\frac1\pi\left( \frac1{\Gamma(r)}\int\limits_0^\infty\frac{t^{r-1}e^{-nt}}{1-e^{-t}}dt+\theta R_{r,n}\right),\quad \theta\in[-1.0],
\end{equation}
where
$$
R_{r,n}=\frac1{n\Gamma(r-1)}\int\limits_0^\infty\frac{t^{r-2}e^{-(n+1)t}}{1-e^{-t}}dt
-\frac1{\Gamma(r)}\int\limits_0^\infty\frac{t^{r-1}e^{-(n+1)t}}{1-e^{-t}}dt=
$$
\begin{equation}
	=\frac1{\Gamma(r)}\left(\frac{r-1}n \int\limits_0^\infty\frac{t^{r-2}e^{-(n+1)t}}{1-e^{-t}}dt
	-\int\limits_0^\infty\frac{t^{r-1}e^{-(n+1)t}}{1-e^{-t}}dt\right).
\end{equation}
However, by virtue of formulas \eqref{1td5} and \eqref{38zh}, the remainder $R_{r,n}$ in formulas \eqref{6g} and \eqref{38d} can also be represented as follows:
\begin{equation}
	R_{r,n}=\frac1n\zeta(r-1,n)-\zeta(r,n)=
	$$
	$$
	=
	\frac1{\Gamma(r)}\left( \frac{r-1}{n}\int\limits_0^\infty\frac{t^{r-2}e^{-nt}}{1-e^{-t}}dt-\int\limits_0^\infty\frac{t^{r-1}e^{-nt}}{1-e^{-t}}dt\right).
\end{equation}

\section*{Acknowledgments.}

This work is a part of the project that has received funding from the EU Horizon 2020 research and innovation programme under the Marie Sklodowska-Curie grant agreement Number 873071, SOMPATY.

This work was also partially supported by the VolkswagenStiftung project ”From Modeling and Analysis to Approximation” and by the grants from the Simons Foundation  (SFI-PD-Ukraine-00014586,~ASS) and (SFI-PD-Ukraine-00014586,~IVS).

\bigskip

CONTACT INFORMATION

\medskip
A.S.~Serdyuk\\01024, Ukraine, Kyiv-4, 3, Tereschenkivska st.\\
serdyuk@imath.kiev.ua

\medskip
I.V.~Sokolenko\\01024, Ukraine, Kyiv-4, 3, Tereschenkivska st.\\
sokol@imath.kiev.ua

\bigskip
	\begin{thebibliography}{99}
	
	\bibitem{Babenko-Pichugov1980} {Babenko V.F., Pichugov S.A.}: Best linear approximation of some classes of differentiable periodic functions. Math. Notes 1980;  27(5): pp.~325–329.	
	
	\bibitem{Efimov1960}
	{Efimov A.V.}: Approximation of continuous periodic functions by Fourier sums. Izv. Akad. Nauk SSSR, Ser. Mat. 1960; 24: pp.~243--296.
	
	\bibitem{Gradshteyn_Ryzhik1963}
	{Gradshteyn I. S., Ryzhik I. M.}: Tables of integrals, sums, series and products. Fizmatgiz, 1963.
	
	\bibitem{Kolmogorov1935}
	{Kolmogoroff A.}: Zur Gr?ssenordnung des Restgliedes Fourierscher Reihen differenzierbarer Funktionen. Ann. of Math. 1935: 36(2): pp.~521--526.
	
	\bibitem{1910-Lebesgue}	{Lebesgue H.}:
	Sur la repr?sentation trigonom?trique approch?e des fonctions satisfaisant ? une condition de Lipschitz. Bulletin de la S. M. F. 1910; 38: pp.~184-210.
	
	\bibitem{Motornyi1974} 
	{Motornyi V.P.}: Approximation of periodic functions by trigonometric polynomials in the mean. Math. Notes 1974; 16(1): pp.~592–599.
	
	\bibitem{Natanson1961}
	{Natanson G.I.}: Approximation by Fourier sums of functions possessing different structural properties on different parts of the domain of definition. Vestn. Leningr. Univ. 1961; 19: pp.~20--35.
	
	\bibitem{Nikol'skii1941}
	{Nikol'skii S.M.}: An asymptotic estimation of the remainder under approximation by Fourier sums.	Dokl. Akad. Nauk SSSR 1941; 32: pp.~386--389.
	
	\bibitem{Nikol'skii1946}
	{Nikol'skii S.M.}: Approximation of functions in the mean by trigonometric polynomials.
	Izv. Akad. Nauk SSSR, Ser. Mat. 1946; 10: pp.~207--256.
	
	\bibitem{Pinkevich1940}
	{Pinkevich V.T.}: On the order of the remainders of the Fourier series of functions differentiable in the sense of Weyl. Izv. Akad. Nauk SSSR, Ser. Mat. 1940; 4: pp.~521--528.
	
	\bibitem{Selivanova1955}
	{Selivanova S.G.}: Approximation by Fourier sums of. the functions possessing a derivative satisfying the Lipschitz condition. Dokl. Akad. Nauk SSSR 1955; 105: pp.~909--912.
	
	\bibitem{Serdyuk_2005_8}
	{Serdyuk A.S.}: Approximation of classes of analytic functions by Fourier sums in the uniform metric. Ukrainian Math. J. 2005; 57(8): pp.~1275--1296. doi:10.1007/s11253-005-0261-0
	
	\bibitem{Serdyuk_2005_10}
	{Serdyuk A.S.}: Approximation of classes of analytic functions by Fourier sums in the metric of the space $L_p$. Ukrainian Math. J. 57(10): pp.~1635--1651. doi:10.1007/s11253-006-0018-4
	
	\bibitem{Serdyuk_Sokolenko2011}
	{Serdyuk A.S., Sokolenko I.V.}: Uniform approximation  of classes of $(\psi,\bar\beta)-$differentiable functions by linear methods. Approximation Theory of Functions and Related Problems, Zb. prac’ Inst. mat. NAN Ukr., Kyiv 2011; 8(1): pp.~181--189.	
	
	\bibitem{Serdyuk_Sokolenko2013}
	{Serdyuk A.S., Sokolenko I.V.}: Approximation by linear methods of classes of $(\psi,\bar\beta)-$differentiable functions. Approximation Theory of Functions and Related Problems, Zb. prac’ Inst. mat. NAN Ukr., Kyiv 2013; 10(1): pp.~245--254.	https://trim.imath.kiev.ua/index.php/trim/article/view/183/149
	
	\bibitem{Serdyuk_Sokolenko2019MFAT}	
	{Serdyuk A.S., Sokolenko I.V.}:  Approximation by Fourier sums in classes of differentiable functions with high exponents of smoothness. Methods of Functional Analysis and Topology 2019; 4: pp.~381--387. https://mfat.imath.kiev.ua/article/?id=1245
	
	\bibitem{Serdyuk_Sokolenko2020UMB} 
	{Serdyuk A.S., Sokolenko I.V.}: Asymptotic Estimates for the Best Uniform Approximations of Classes of Convolution of Periodic Functions of High Smoothness. Journal of Mathematical Sciences, Vol. 252, pp.~526–540 (2021). doi:10.1007/s10958-020-05178-1 
	
	
	\bibitem{Serdyuk_Sokolenko2022UMJ} 
	{Serdyuk A.S., Sokolenko I.V.}: Approximation by Fourier Sums in the Classes of Weyl–Nagy Differentiable Functions with High Exponent of Smoothness. Ukrainian Math. J. 2022; 74(5): pp.~783--800. doi:10.1007/s11253-022-02101-6.
	
	
	
	\bibitem{Serdyuk_Stepanyuk2019Anal} 
	{Serdyuk A.S., Stepanyuk T.A.}: Uniform Approximations by Fourier Sums in Classes of Generalized Poisson Integrals. Analysis Math. 2019; 45(1): pp.~201–236. doi:10.1007/s10476-018-0310-1
	
	\bibitem{Serdyuk_Stepanyuk2023UMJ} 
	{Serdyuk A.S., Stepanyuk T.A.}: Uniform approximations by Fourier sums on the sets of convolutions of periodic functions of high smoothness. Ukrainian Math. J.  2023; 75(4):  pp. 542--567. doi:10.37863/umzh.v75i4.7411
	
	\bibitem{Sokolov1955}
	{Sokolov I.G.}: The remainder term of the Fourier series of differentiable functions. 
	Dokl. Akad. Nauk SSSR 1955; 103, pp.~23--26.
	
	\bibitem{Stepanets1995}
	{Stepanets A.I.}:
	Classification and Approximation of Periodic Functions,
	Kluwer Academic Publishers,
	Dordrecht,
	1995.	
	
	\bibitem{Stepanets2005}
	{Stepanets A.I.}: Methods of Approximation Theory. 
	Utrecht, VSP, 2005.
	
	\bibitem{Stechkin1956}
	{Stechkin S.B.}: On the best approximation of certain classes of periodic functions by trigonometric polynomials. Izv. Akad. Nauk SSSR, Ser. Mat. 1956; 20: pp.~643--648.
	
	\bibitem{Stechkin1980}
	{Stechkin S.B.}: Estimation of the remainder for the Fourier series for differentiable functions. Proc. Steklov Inst. Math. 1980; 145: pp.~139--166.
	
	\bibitem{Telyakovskii1968}
	{Telyakovskii S.A.}: Approximation of differentiable functions by the partial sums of their Fourier series. Math. Notes 1968; 4: pp.~668--673.
	
	\bibitem{Telyakovskii1989}
	{Telyakovskii S.A.}: Approximation of functions of high smoothness by Fourier sums.
	Ukrainian Math. J.  1989; 41(4): pp.~444--451. https://doi.org/10.1007/BF01060623
	
	\bibitem{Telyakovskii1994} 
	{Telyakovskii S.A.}: On approximation by Fourier sums of differentiable functions of high smoothness. Proc. Steklov Inst. Math. 1994; 1(198): pp.~183–201.
	
\end{thebibliography}
\end{document}